\documentclass[12pt, twoside, leqno]{article}

\usepackage{amsfonts}
\usepackage{amsmath}
\usepackage{amssymb}
\usepackage{amsthm}
\usepackage{ifthen}
\usepackage{enumerate}

\newtheorem{theorem}{Theorem}[section]
\newtheorem{proposition}[theorem]{Proposition}

\newtheorem{lemma}[theorem]{Lemma}
\newtheorem*{structconj}{Structural Conjecture}
\theoremstyle{remark}

\newtheorem{example}[theorem]{Example}
\theoremstyle{definition}

\newcommand{\bcdot}{$\discretionary{\mbox{$ \cdot $}}{}{}$}
\newcommand{\parder}[3][Default]{
	\frac{\partial \ifthenelse{\equal{#1}{Default}}{}{^{#1}}#2}{
              \partial #3 \ifthenelse{\equal{#1}{Default}}{}{^{#1}}}}

\newcommand{\jac}{{\mathcal J}}

\newcommand{\GL}{\operatorname{GL}}
\newcommand{\imp}{{\mathversion{bold}$\Rightarrow$} }
\newcommand{\C}{{\mathbb C}}

\newcommand{\N}{{\mathbb N}}

\newcommand{\K}{{K}}
\newcommand{\Kbar}{\bar{K}}
\newcommand{\tp}{^{\mathrm t}}
\newcommand{\I}{{\mathrm i}}

\newcommand{\rk}{\operatorname{rk}}
\newcommand{\Mat}{\operatorname{Mat}}

\title{Triangularization properties of power linear maps and the Structural Conjecture}
\author{Michiel de Bondt\footnote{The first author was supported by the Netherlands 
                          Organisation for Scientific Research (NWO).} \\
Department of Mathematics, Radboud University \\ 
Nijmegen, The Netherlands \\
\emph{E-mail:} M.deBondt@math.ru.nl
\and
Dan Yan \\
School of Mathematical Sciences, Graduate University of \\
Chinese Academy of Sciences, Beijing 100049, China \\
\emph{E-mail:} yan-dan-hi@163.com}
\date{}

\begin{document}

\maketitle 

\renewcommand{\thefootnote}{}

\footnote{2010 \emph{Mathematics Subject Classification}: Primary 14R10; Secondary 14R15.}

\footnote{\emph{Key words and phrases}: Jacobian Conjecture, Dru{\.z}kowski map, Structural Conjecture, 
Linearly triangularizable.}

\renewcommand{\thefootnote}{\arabic{footnote}}
\setcounter{footnote}{0}

\begin{abstract}
\noindent
In this paper, we discuss several additional properties a power linear Keller map may have.
The Structural Conjecture by Dru{\.z}kowski in \cite{MR0714105} asserts that two such properties are 
equivalent, but we show that one of this properties is stronger than the other. We even show 
that the property of linear triangularizability is strictly in between. Furthermore, we give some
positive results for small dimensions and small Jacobian ranks.
\end{abstract}

\section{Introduction}

Throughout this paper, we will write $\K$ for any field of characteristic zero, $\Kbar$ for its algebraic 
closure, and $\K[x]=\K[x_1,x_2,\ldots,x_n]$ for the polynomial algebra over $\K$ with $n$ indeterminates 
$x=x_1,x_2,\ldots,x_n$. Let $F=(F_1,F_2,\ldots,F_n): \K^n \rightarrow \K^n$
be a polynomial map, that is, $F_i\in\K[x]$ for all $1\leq i\leq n$, or briefly $F \in \K[x]^n$.
We view $F$ and $x$ as column matrices, as well as $\partial = \partial_1,\partial_2,\ldots,\partial_n$,
where $\partial_i = \parder{}{x_i}$. Just like in $F=(F_1,F_2,\ldots,F_n)$, we see any other
tuple whose elements are separated by commas as a column vector as well.
Let $M\tp$ be the transpose of a matrix $M$ and write
$$
\jac F = \big(\partial (F\tp)\big)\tp = \left( \begin{array}{cccc}
\partial_1 F_1 & \partial_2 F_1 & \cdots &  \partial_n F_1 \\
\partial_1 F_2 & \partial_2 F_2 & \cdots &  \partial_n F_2 \\
\vdots & \vdots & \ddots & \vdots \\
\partial_1 F_n & \partial_2 F_n & \cdots &  \partial_n F_n
\end{array} \right) 
$$
We say that a polynomial map $F$ is a Keller map if $\det \jac F \in K^{*}$.
The well-known Jacobian Conjecture, raised by O.-H. Keller in 1939 in \cite{MR1550818}, states that a 
polynomial map $F:\K^n \rightarrow \K^n$ is invertible if it is a Keller map.
This conjecture is still open for all $n \ge 2$.
In \cite[Th.\@ 3]{MR0714105}, Ludwik Dru{\.z}kowski showed that it suffices to consider polynomial maps 
$F: \C^n \rightarrow \C^n$ of the form $F = x + (Ax)^{*3}$, where $A \in \Mat_n(\C)$ and $M^{*d}$
is the $d$-th Hadamard power (repeated Hadamard product with itself) of a matrix $M$.

In the same paper, Dru{\.z}kowski also formulated the Structural Conjecture, which asserts the following.
Write $M|_{x=G}$ for the substitution of $x$ by $G$ in a matrix $M$.

\begin{structconj}
If $F = x + (Ax)^{*3}$ and $\det \jac F = 1$, then the following conditions are equivalent.
\begin{description}

\item[\bf(JC)] $\det\big((\jac F)|_{x=v_1} + (\jac F)|_{x=v_2}\big) \ne 0$ for all $v_1, v_2 \in \C^n$.

\item[\bf(**)] There exist $b_i \in \C^n$ and $c_j \in \C^n$ such that $c_j\tp b_i = 0$ 
for every $i \ge j \ge 1$, and $F$ has the form $x + \sum_{i=1}^{n-1} (c_i\tp x)^3 b_i$.

\end{description}
\end{structconj}

\noindent
Actually, Dru{\.z}kowski writes $F = x + \sum_{j=1}^n (a_j\tp x)^3 e_j$ instead of $F = x + (Ax)^{*3}$, where
$e_j$ is the $j$-th standard basis unit vector. Hence $a_j\tp$ corresponds to the $j$-th row $A_j$ of $A$.
Since the vectors $c_j$ and $b_i$ are viewed as column matrices, the matrix product $c_j\tp b_i$ has only
one entry, which we see as an element of $\C$.

We call a polynomial map $F$ over $\K$ {\em linearly triangularizable} if there exists a $T \in \GL_n(\K)$
such that the Jacobian of $T^{-1} F(Tx)$ is a triangular matrix. For Keller maps of the form $F = x + H$ with
$H$ homogeneous of degree $d \ge 2$, the existence of such a $T$ automatically means that the diagonal of 
$$
\jac \big(T^{-1} F(Tx)\big) = T^{-1} (\jac H)|_{x=Tx} T
$$
is zero, because $\jac H$ has to be nilpotent due to the Keller condition.

We embed the Structural Conjecture in a more general scope, where $F$ has the form $x + H$
such that $\jac H$ is nilpotent, and compare its conditions with linear 
triangularizability and other properties. We give positive results in special cases and counterexamples
in general. When we give counterexamples, we will give one of the form $F = x + H$ with
$H$ homogeneous of degree $d$ and one of the form $F = x + (Ax)^{*d}$, for every $d \ge 3$ and possibly also 
for $d = 2$.

\section{Triangularization}

In the following proposition, the conditions (JC) and (**) of the Structural Conjecture are included
in a chain of six properties. Furthermore, we generalize to maps $x + H$ such that $H$ has no constant
terms instead of being homogeneous.

\begin{proposition} \label{chain6}
Let $F = x + H$ be any polynomial map of degree $d$ over $\K$. Then for
\begin{align} 
&F \mbox{ is invertible} 
\label{jc} \tag{JC$^{-}$\hspace*{-0.5pt}} \\
&\det\Big(\sum_{i=1}^{d-1} \jac F|_{x=v_i}\Big) \in \Kbar^{*} \mbox{ for all $v_i \in \Kbar^n$} 
\label{adddruz} \tag{JC} \\
&\det\Big(\sum_{i=1}^n \jac F|_{x=v_i}\Big) \in \Kbar^{*} \mbox{ for all $v_i \in \Kbar^n$} 
\label{addkel} \tag{JC$^{+}$\hspace*{-0.5pt}}
\end{align}
and for the existence of $b_i \in \K^n$, $c_j \in \K^n$ and $d_i \in \{1,2,\ldots,d\}$ such that
$c_j\tp b_i = 0$ for every $i \ge j \ge 1$, and
\begin{align}
H &= \sum_{i=1}^N (c_i\tp x)^{d_i} b_i \mbox{ for some } N \in \N 
\label{yantri} \tag{*} \\
H &= \sum_{i=1}^{n-1} (c_i\tp x)^{d_i} b_i 
\label{druztri} \tag{**} \\
H &= \sum_{i=1}^{n-1} (c_i\tp x)^{d_i} b_i \mbox{ and } b_1,b_2,\ldots,b_{n-1} \mbox{ are linearly independent}
\label{dittotri} \tag{***}
\end{align}
we have $\eqref{jc} \Leftarrow \eqref{adddruz} \Leftarrow \eqref{addkel} \Leftarrow 
\eqref{yantri} \Leftarrow \eqref{druztri} \Leftarrow \eqref{dittotri}$.

Furthermore, \eqref{jc}, \eqref{adddruz} and \eqref{addkel} are satisfied when $d \le 2$.
\end{proposition}
 
\begin{proof}
Notice that the last two implications are trivial. The first two implications
follow from \cite[Cor.\@ 2.3]{MR2948624} and \cite[Th.\@ 3.5]{MR2948624} respectively. 
The last claim follows from the first two implications and \cite[Prop.\@ 3.1]{MR2948624}. 

To show the third implication, assume that \eqref{yantri} holds and take $v_1, v_2, \ldots, 
\allowbreak v_n \in \Kbar$ arbitrary. Then 
$$
S := \sum_{k=1}^n (\jac H)|_{x=v_k} = \sum_{k=1}^n \sum_{i=1}^N b_i \cdot d_i (c_i\tp v_k)^{d_i-1} \cdot 
c_i\tp = \sum_{i=1}^N b_i \Big(d_i \sum_{k=1}^n(c_i\tp v_k)^{d_i-1}\Big) c_i\tp 
$$
It follows that in the expansion of $S^{N+1}$, each term will have a factor
$c_j\tp \cdot b_i$ such that $i \ge j \ge 1$, which is zero by assumption. Hence $S^{N+1} = 0$. 
Thus $S$ is nilpotent and $\det(\sum_{i=1}^n \jac F|_{x=v_i}) = \det (n I_n + S) = n^n \in \Kbar^{*}$. 
\end{proof}

\noindent
In the last section, we will show that $\eqref{jc} \nRightarrow \eqref{adddruz}$ and
$\eqref{addkel} \nRightarrow \eqref{yantri} \nRightarrow \eqref{druztri} \nRightarrow 
\eqref{dittotri}$, even in the case where $H$ is {\em homogeneous power linear}, i.e.\@ $H = \sum_{i=1}^n 
(c_i\tp x)^d e_i$ for some $c_i \in \K^n$ and a $d \ge 1$. But first, we formulate a lemma and a theorem
about the starred equations. We call $H$ {\em non-homogeneous power linear} if $H = \sum_{i=1}^n 
(c_i\tp x)^{d_i} e_i$ for some $c_i \in \K^n$ and some $d_i \ge 1$.

\begin{lemma} \label{termtri}
Let $N \in \N$ and suppose that there exist $b_i \in \K^n$, $c_j \in \K^n$ and $d_i \in \N$ such that
$$
H = \sum_{i=1}^N (c_i\tp x)^{d_i} b_i
$$
Then the following statements are equivalent.
\begin{enumerate}[\upshape (i)]

\item There exists a $\sigma \in S_N$ such that $c_{\sigma(j)}\tp b_{\sigma(i)} = 0$
for every $i \ge j \ge 1$.

\item There exists a $T \in \GL_n(\K)$
such that the Jacobian of $T^{-1} (c_i\tp Tx)^{d_i} b_i$ is lower triangular
with zeroes on the diagonal for all $i \le N$.

\end{enumerate}
Furthermore, if $b_1, b_2, \ldots, b_N$ are linearly independent and $\sigma$ satisfies 
{\upshape (i)}, then we can choose the $T \in \GL_n(\K)$ which satisfies 
{\upshape (ii)} such that
\begin{equation} \label{dittox}
b_{\sigma(i)} = T e_{n - N + i}
\end{equation}
for each $i$.
\end{lemma}

\begin{proof}
We prove that (i) and (ii) are equivalent, showing the last claim along the road.
\begin{description}
 
\item[(ii) \imp (i)]
Suppose that (ii) holds. Let $m_j$ be the number of trailing coordinates zero of $T\tp c_j$ for each $j$. 
By reordering terms of $H$, we can obtain that $m_j \ge m_i$ for each $i \ge j \ge 1$. By (ii), 
\begin{equation} \label{jacterm}
\jac \big(T^{-1} (c_i\tp T x)^d b_i\big) = T^{-1} b_i \cdot d_i (c_i\tp T x)^{d_i-1} \cdot c_i\tp T 
\end{equation}
is lower triangular with zeroes on the diagonal. 
Hence the number of leading coordinates zero of $T^{-1} b_i$ is at least $n - m_i \ge n - m_j$ for each 
$i \ge j \ge 1$. Comparing the numbers of leading and trailing coordinates zero, we get 
$c_j\tp b_i = c_j\tp T \cdot T^{-1} b_i = 0$ for all $i \ge j \ge 1$, which is (i) with $\sigma = 1$. 
So we can take $\sigma = 1$ when $m_j \ge m_{j+1}$ for each $j$ already before 
reordering the terms of $H$.

\item[(i) \imp (ii)]
Suppose that (i) holds. Again by reordering terms of $H$, we can obtain that $\sigma = 1$. 
Suppose that the vector space spanned by the column vectors $b_1, b_2, \ldots, b_N$
has dimension $r$. Then there are $\tau(1) < \tau(2) < \cdots < \tau(r)$ such that 
$b_{\tau(1)}, b_{\tau(2)}, \ldots, b_{\tau(r)}$ is a basis of this vector space. Now choose
$\tau(1) + \tau(2) + \cdots + \tau(r)$ as large as possible. Then $b_i$ is linearly dependent of 
$b_{\tau(k)}, b_{\tau(k+1)},\ldots,b_{\tau(r)}$ for all $k$ and all $i > \tau(k-1)$, where $\tau(0) = 0$
and where zero vectors are linearly dependent of the empty set.
Furthermore, $c_i\tp b_{\tau(k)} = c_i\tp b_{\tau(k+1)} = \cdots = c_i\tp b_{\tau(r)} = 0$ for all $k$ and all
$i \le \tau(k)$ on account of (i) with $\sigma = 1$.

Take $T \in \GL_n(\K)$ such that the last $r$ columns of $T$ are $b_{\tau(1)}, b_{\tau(2)}, \ldots, 
\allowbreak b_{\tau(r)}$, in that order. Then we have \eqref{dittox} with $\sigma = 1$
if $b_1, b_2, \ldots, b_N$ are linearly independent. Take $i \le N$ arbitrary. It suffices to show that
\eqref{jacterm} is lower triangular with zeroes on the diagonal. This is trivial when $b_i = 0$,
so assume that $b_i \ne 0$.
Then by definition of $r$ and $\tau$, there exists a $k \ge 1$ such that $\tau(k) \geq i > \tau(k-1)$.
As we have seen above, $b_i$ is linearly dependent of $b_{\tau(k)}, b_{\tau(k+1)}, \ldots, b_{\tau(r)}$ and
$c_i\tp b_{\tau(k)} = c_i\tp b_{\tau(k+1)} = \cdots = c_i\tp b_{\tau(r)} = 0$.

Hence $T^{-1} b_i$ is linearly dependent of $e_{n-r+k}, e_{n-r+k+1}, \ldots, e_n$ and
$c_i\tp T \bcdot e_{n-r+k} = c_i\tp T e_{n-r+k+1} = \cdots = c_i\tp T e_n = 0$ by definition of $T$.
Consequently, all nonzero entries of \eqref{jacterm} are within the 
submatrix consisting of rows $n-r+k, \allowbreak n-r+k+1, \ldots, n$ and columns $1,2,\ldots,n-r+k-1$ of it. 
Since $i$ was arbitrary, we obtain (ii). \qedhere

\end{description}
\end{proof}

\begin{theorem} \label{tri}
Let $x + H$ be any map of degree $d \ge 1$ over $\K$ such that $H(0) = 0$. Then we have the following.
\begin{enumerate}[\upshape (i)]

\item $H$ is of the form \eqref{yantri}, if and only if there exists a $T \in \GL_n(\K)$
such that the Jacobian of $T^{-1} H(Tx)$ is lower triangular with zeroes on the diagonal, i.e.\@
$H$ is linearly triangularizable and $\jac H$ is nilpotent.

\item $H$ is of the form \eqref{druztri}, if and only if there exists $b_i, c_j \in \K^n$ and 
there exists a $T \in \GL_n(\K)$ such that $H = \sum_{i=1}^{n-1} (c_i\tp x)^{d_i} b_i$ 
and the Jacobian of $T^{-1} (c_i\tp Tx)^{d_i} b_i$ is lower triangular with zeroes on the diagonal
for all $i \le n-1$.

\item $H$ is of the form \eqref{dittotri}, if and only if there exists a $T \in \GL_n(\K)$
such that each component of $T^{-1} H(Tx)$ is a power of a linear form and the Jacobian of 
$T^{-1} H(Tx)$ is lower triangular with zeroes on the diagonal.

\end{enumerate}
\end{theorem}

\begin{proof}
Since the three results have similarities, we structure the proof as follows.
\begin{description}
 
\item[Only-if-parts.] All only-if-parts follow immediately from (i) $\Rightarrow$ (ii) in 
lem\-ma \ref{termtri}, except the claim that each component of $T^{-1} H(Tx)$ is a power of a linear form
in (iii). So assume that $H$ is of the form \eqref{dittotri}. By \eqref{dittotri} and \eqref{dittox},
we have that
$$
T^{-1} H(Tx) = \sum_{i=1}^{n-1} T^{-1} (c_{\sigma(i)}\tp T x)^{d_i} b_{\sigma(i)} = 
\sum_{i=1}^{n-1} (c_{\sigma(i)}\tp T x)^{d_i} e_{i+1}
$$
for some $\sigma \in S_{n-1}$. So $T^{-1} H(Tx)$ is of the desired form.

\item[If-parts.] The if-part of (ii) follows immediately from (ii) $\Rightarrow$ (i) of 
lemma \ref{termtri}. To prove the if-part of (i), suppose that $T^{-1} H(Tx)$ has a lower 
triangular Jacobian with zeroes on the diagonal. Then there exists an $r \in \N$, such 
that we can write the $(i+1)$-th 
component of $T^{-1} H (Tx)$ as a linear combination of $r$ powers of linear forms $c_j\tp Tx$ 
in $\K[x_1,x_2,\ldots,x_i]$, for each $i \ge 1$. Furthermore, the first component of $T^{-1} H (Tx)$ is zero
on account of $H(0) = 0$. Hence we have
\begin{align*}
T^{-1} H(Tx) &= \sum_{i=1}^{n-1} \sum_{j=1}^r \big(c_{r(i-1)+j}\tp T x\big)^{d_{r(i-1)+j}} e_{i+1} \\
&= T^{-1} \sum_{i=1}^{n-1} \sum_{j=1}^r \big(c_{r(i-1)+j}\tp T x\big)^{d_{r(i-1)+j}} T e_{i+1}
\end{align*}
Taking $b_{r(i-1)+j} = T e_{i+1}$ for all $i$ and for all $j$ with $1 \le j \le r$, 
we have
$$
H = \sum_{i=1}^{n-1} \sum_{j=1}^r (c_{r(i-1)+j}\tp x)^{d_{r(i-1)+j}} b_{r(i-1)+j} 
= \sum_{i=1}^{r(n-1)} (c_i\tp x)^{d_i} b_i 
$$
Furthermore, for each $j$, the Jacobian of $(c_j\tp T x)^{d_j} T^{-1} b_j$ only has 
non\-zero entries within the submatrix consisting of row $i+1$ and columns 
$1,2,\ldots,i$ of it by definition of $c_j$ and $b_j$, where $i = \lceil j/r \rceil$. Hence 
the Jacobian of $(c_j\tp T x)^{d_j} T^{-1} b_j$ is 
lower triangular with zeroes on the diagonal for all $j$. Now the if-part of (i) follows from (ii) 
$\Rightarrow$ (i) of lemma \ref{termtri}. The if-part of (iii) follows as well, because we can take
$r = 1$ in that case, so that $r(n-1) = n-1$ and the $b_i$'s are linearly independent. \qedhere

\end{description}
\end{proof}

\section{Positive results}

\noindent
First, we formulate a theorem about maps $x + H$ such that $H$ is homogeneous and $\jac H$ is nilpotent.

\begin{theorem} \label{posres}
Assume that $H \in \K[x]^n$ is homogeneous of degree $d \ge 1$, such that $\jac H$ is nilpotent. 
Then we have \eqref{dittotri} (and hence five $\Rightarrow$'s) if $n \le 2$, 
and \eqref{yantri} (and hence three $\Rightarrow$'s) if $n = 3$ or $n = 4 = d+2$. 
Furthermore, the implication chain $\eqref{adddruz} \Rightarrow \eqref{addkel}
\Rightarrow \eqref{yantri}$ holds when $n = 4 = d+1$.

If $H$ is power linear in addition, then the above claims even hold when we replace the estimates on $n$
by estimates on $\rk \jac H$.
\end{theorem}

\begin{proof}
We show the equivalent properties in (i) and (iii) of theorem \ref{tri} respectively 
instead of \eqref{yantri} and \eqref{dittotri}. We start with the cases where $H$ is only homogeneous.

The case $n \le 2$ follows from \cite[Lem.\@ 3]{MR2108467} and the case $n = 3$ follows from 
\cite[Th.\@ 1.1]{MR2179727}. 
The case $n = 4 = d+2$ follows from a corresponding strong nilpotence result in \cite{MR1129181}, and 
the equivalence of strong nilpotence and the property in (i) of theorem \ref{tri}, which is proved in 
\cite{MR1412743}.
The case $n = 4 = d+1$ follows from \cite[Th.\@ 4.6.5]{homokema} and the fact that $F = x+H$, with $H$ 
as in \cite[Th.\@ 4.6.5]{homokema}, does not 
satisfy \eqref{adddruz}, because the rightmost two 
columns of $(\jac F)|_{x=(1,\I,0,0)} + (\jac F)|_{x=(1,-\I,0,0)}$ are equal, where
$H = \big(0, \lambda x_1^3, x_2(x_1 x_3-x_2 x_4) + p(x_1,x_2), x_1(x_1 x_3-x_2 x_4) + q(x_1,x_2)\big)$.

Assume from now on that $H$ is power linear in addition.
The case $\rk \jac H \le 2$ follows from \cite[Th.\@ 4.7]{MR2333192}, because $\K = \C$
is not used in its proof, or theorem \ref{posresnh} below. The cases $\rk \jac H = 3$ and
$\rk \jac H = 4 = d+2$ follow by way of \cite[Th.\@ 2]{MR2108467} from the cases $n = 3$ and $n = 4 = d+2$ 
respectively. The case $\rk \jac H = 4 = d+1$ follows from the case $n = 4 = d+1$ by way of a variant 
of \cite[Th.\@ 2]{MR2108467}, namely with \eqref{yantri} replaced by $\eqref{adddruz} \Rightarrow 
\eqref{yantri}$. To prove this variant, one can follow the proof of \cite[Th.\@ 2]{MR2108467},
to see that it suffices to show that $F_1 = T^{-1} \circ F \circ T$ satisfies \eqref{adddruz}
if $F$ does, in that proof.
\end{proof}

\noindent
In theorem \ref{posresnh} below, which is the non-homogeneous variant of theorem \ref{posres}, we must 
replace the estimates on $n$ and $\rk \jac H$ of theorem \ref{posres} by estimates on $n + 1$ and 
$\rk \jac H + 1$ respectively, except the estimate $n \le 2$ for $\eqref{druztri} \Rightarrow 
\eqref{dittotri}$, and the estimate $\rk \jac H \le 2$ for $\eqref{dittotri}$, which can be maintained.

\begin{theorem} \label{posresnh}
Assume that $H \in \K[x]^n$ has degree $d$, such that $H(0) = 0$ and $\jac H$ is nilpotent. 
Then we have \eqref{dittotri} (and hence five $\Rightarrow$'s) if $n \le 1$, both
\eqref{yantri} and $\eqref{druztri} \Rightarrow \eqref{dittotri}$ (and hence four $\Rightarrow$'s) 
if $n = 2$, and \eqref{yantri} (and hence three $\Rightarrow$'s) if $n = 3 = d+1$. 
Furthermore, the implication chain $\eqref{adddruz} \Rightarrow \eqref{addkel} 
\Rightarrow \eqref{yantri}$ holds when $n = 3 = d$.

If $H$ is power linear in addition, then the above claims even hold when we replace the estimates on $n$
by estimates on $\rk \jac H$, and additionally \eqref{dittotri} (and hence five $\Rightarrow$'s) holds
when $\rk \jac H = 2$.

Furthermore, if we replace \eqref{yantri} and \eqref{dittotri} by their equivalences in 
{\upshape (i)} and {\upshape (iii)} of theorem \ref{tri}, then the condition $H(0) = 0$ is no longer necessary.
\end{theorem}

\begin{proof}
We show the equivalent properties in (i) and (iii) of theorem \ref{tri} respectively 
instead of \eqref{yantri} and \eqref{dittotri}. We start with the cases where $H$ only has a nilpotent
Jacobian.

The case $n = 1$ is trivial, because $H = 0$ in that case.
Notice that in the cases $n = 2$ and $n = 3 = d+1$, the homogeneization
$x_{n+1}^d H(x_{n+1}^{-1}x, 0)$ of $H$ has a strongly nilpotent Jacobian on account of
theorem \ref{posres}. By substituting $x_{n+1} = 1$, we see that the Jacobian of $H$ itself
is strongly nilpotent as well. By the equivalence of strong nilpotence and the property in (i) of 
theorem \ref{tri}, which is proved in \cite{MR1412743}, we have the property in (i) of theorem 
\ref{tri}, and hence also \eqref{yantri}, when $n = 2$ or $n = 3 = d+1$. 
This gives the case $n = 3 = d+1$, and also the case $n = 2$, because \eqref{druztri}
and \eqref{dittotri} are trivially equivalent when $n = 2$. 

In order to prove the case $n = 3 = d$, assume that $H$ does not have
the property in (i) of theorem \ref{tri}. By \cite[Cor.\@ 4.6.6]{homokema}, we may assume that 
the first components of $T^{-1} H(Tx)$ equals $\lambda \in \K$ for some $T \in \GL_3(\K)$. Following the 
proof of \cite[Th.\@ 4.6.5]{homokema}, we see that $T^{-1} H(Tx) = 
\big(\lambda, x_1\bcdot (x_2 - x_1 x_3) + p(x_1),(x_2 - x_1 x_3) + q(x_1)\big)$ for some $T \in \GL_3(\K)$.
Since the rightmost two columns of $(\jac F)|_{x=(\I,0,0)} + (\jac F)|_{x=(-\I,0,0)}$ are equal, we see
that \eqref{adddruz} does not hold, as desired.

Assume from now on that $H$ is (non-homogeneous) power linear in addition. The cases
$\rk \jac H = 3 = d$ and $\rk \jac H = 3 = d+1$ follow in a similar manner as the cases 
$\rk \jac H = 4 = d+1$ and $\rk \jac H = 4 = d+2$ respectively
in theorem \ref{posres}. So assume that $\rk \jac H \le 2$. Take $\lambda$ and $\mu$
as in lemma \ref{lambdamu} below. If $\mu\tp H$ is a power of a linear
form, then we take $T \in \GL_n(\K)$ such that $\lambda\tp$ and $\mu\tp$ are the first two rows of 
$T^{-1}$, in that order, and for the remaining rows of $T^{-1}$ we transpose standard basis unit vectors.
Since $\lambda\tp$ and $\mu\tp$ generate the row space of $\jac H$, we see that 
$\lambda\tp T = e_1$ and $\mu\tp T = e_2$ generate the row space of $\jac (T^{-1}H(Tx))$. 
Using additionally that $\mu\tp H \in \K[\lambda\tp x]$, we obtain that $T^{-1} H (Tx) \in 
\K \times \K[x_1] \times K[x_1,x_2]^{n-2}$ has a lower triangular Jacobian with zeroes on the diagonal. 
So we have (i) of theorem \ref{tri} and hence also \eqref{yantri}.

So assume that $\mu\tp H$ is not a power of a linear form. By lemma \ref{lambdamu} below, we have
\begin{equation} \label{muH}
\mu\tp H = \nu_1 (\lambda\tp x)^{d_1} + \nu_2 (\lambda\tp x)^{d_2} + \cdots + \nu_r (\lambda\tp x)^{d_r} 
\end{equation}
where $r \ge 2$, $\nu \in (\K\setminus \{0\})^r$ and $\{0,1\} \ni \lambda\tp H \le d_1 < d_2 < \cdots < d_r$. 
Take $T \in \GL_n(\K)$ such that $\lambda\tp$ is the first row of $T^{-1}$. Let $V \in \Mat_{r,n}(\{0,1\})$ 
such that $V_{ij} = 1$, if and only if $\deg H_j = d_i$. Without worrying about linear
independence of rows at this stage, take for each $i$ with $2 \le i \le r$, 
the $(i+1)$-th row of $T^{-1}$ equal to $T^{-1}_{i+1} = \nu_i^{-1} \mu\tp * V_i$, 
where $*$ is the Hadamard product and $V_i$ is the $i$-th row of $V$. Then by definition of $V_i$,
the $(i+1)$-th component $T^{-1}_{i+1} H$ of $T^{-1} H$ is $\nu_i^{-1}$ times the 
homogeneous part of degree $d_i$ of \eqref{muH}, which is $\nu_i^{-1} \nu_i (\lambda\tp x)^{d_i} = 
(\lambda\tp x)^{d_i}$, for each $i$ with $2 \le i \le r$.

Still without worrying about linear independence of rows, take the second row of $T^{-1}$ equal to
$$
T^{-1}_2 = \nu_1^{-1} \big(\mu\tp - \big(\mu\tp * (V_2 + \cdots + V_r) \big) \big)
$$
Since $\mu\tp * V_i = \nu_i T^{-1}_{i+1}$ for each $i$ with $2 \le i \le r$ by definition of 
$T^{-1}$, we have
\begin{equation} \label{summu} 
T^{-1}_2 = \nu_1^{-1}\big(\mu\tp  - (\nu_2 T^{-1}_3 + \cdots + \nu_r T^{-1}_{r+1})\big)
\end{equation}
and the second component of $T^{-1} H$ equals 
\begin{align*}
T^{-1}_2 H
&= \nu_1^{-1}\big(\mu\tp H - (\nu_2 T^{-1}_3 H + \cdots + \nu_r T^{-1}_{r+1} H)\big) \\
&= \nu_1^{-1}\big(\mu\tp H - (\nu_2 (\lambda\tp x)^{d_2} + \cdots + \nu_r (\lambda\tp x)^{d_r})\big)
\end{align*}
which by \eqref{muH} is equal to $\nu_1^{-1} \nu_1 (\lambda\tp x)^{d_1} = (\lambda\tp x)^{d_1}$. 
Thus for each $i \in \{1,2,\ldots,\allowbreak r\}$, the $(i+1)$-th component of $T^{-1} H$ is equal to
$(\lambda\tp x)^{d_i}$. 

Since $\{0,1\} \ni \lambda\tp H \le d_1 < d_2 < \cdots < d_r$, we have 
$\deg T^{-1}_1 H = \deg \lambda\tp H < \lambda\tp H$
and the degrees of the first $r+1$ components of $T^{-1} H$ are strictly increasing. 
Hence by $ T^{-1}_1 = \lambda\tp \ne 0$, the first $r+1$ rows of 
$T^{-1}$ are indeed linearly independent. Take transposed standard basis unit vectors for the remaining rows of 
$T^{-1}$. By $\lambda\tp T = e_1\tp$, we get that the first $r+1$ components of $T^{-1} H(Tx)$
are $\lambda\tp H, x_1^{d_1}, x_1^{d_2}, \ldots, x_1^{d_r}$, so
$T^{-1} H(Tx) \in \K \times \K[x_1]^r \times \K[x]^{n-r-1}$.
Furthermore, we see that $T^{-1} H(Tx)$ is power linear.

By \eqref{summu}, we have $\sum_{i=1}^r \nu_i T^{-1}_{i+1} = \mu\tp$,
so $\mu\tp$ is a linear combination of the first $r+1$ rows of $T^{-1}$.
Since $\lambda\tp$ and $\mu\tp$ generate the row space of $\jac H$ and are linear combinations of the 
first $r+1$ rows of $T^{-1}$, we see that $\lambda\tp T$ and $\mu\tp T$ generate the row space of 
$(\jac H) \cdot T$ and are linear combinations of $e_1\tp, e_2\tp, \ldots, e_{r+1}\tp$. From the 
fact that $H$ is (non-homogeneous) power linear, we can deduce that the row space of 
$(\jac H) \cdot T$ is the same as that of $\jac \big(T^{-1} H(Tx)\big)$.
Hence $T^{-1} H(Tx) \in \K[x_1,x_2,\ldots, x_{r+1}]^n$. Since we
obtained above that $T^{-1} H(Tx) \in \K \times \K[x_1]^r \times 
\K[x]^{n-r-1}$ as well, we can deduce that $T^{-1} H(Tx) \in 
\K \times \K[x_1]^r \times \K[x_1,x_2,\ldots, x_{r+1}]^{n-r-1}$.
Hence $T^{-1}H(Tx)$ has a lower triangular Jacobian with
zeroes on the diagonal. So we have (i) of theorem \ref{tri} and hence also \eqref{yantri}.
\end{proof}

\begin{lemma} \label{lambdamu}
Assume that $H \in \K[A_1 x, A_2 x, \ldots, A_n x]^n$, where $A_j$ is the $j$-th row of a matrix 
$A \in \Mat_n(\K)$ such that $\rk A \le 2$ and $\jac H$ is nilpotent. Then there exists linearly independent 
$\lambda, \mu \in K^n$, such that $\mu\tp H \in \K[\lambda\tp x]$ has no terms of degree less than 
$\lambda\tp H \in \{0,1\}$, and $\lambda\tp$ and $\mu\tp$ generate the row space of $A$.
\end{lemma}

\begin{proof}
Using the case $n = 2$ of theorem \ref{posresnh} (instead of the case $n = 3$ of theorem \ref{posres}), 
we obtain by similar techniques as in the proof of the 
case $\rk \jac H = 3$ of theorem \ref{posres} that there exists a $T \in \GL_n(\K)$ such that 
$ATx \in \K[x_1,x_2]^n$ and the Jacobian of $T^{-1} H(Tx) \in \K[x_1,x_2]^n$ is lower triangular with
zeroes on the diagonal. 
By a subsequent linear conjugation on the first two coordinates, we can even obtain in addition that
the first component of $T^{-1} H(Tx)$ is contained in $\{0,1\}$, and that the second component of 
$T^{-1} H(Tx)$ has no constant term if the first component already has. 

Now take for $\lambda\tp$ the first row of $T^{-1}$ and for $\mu\tp$ the second row of $T^{-1}$.
Then $\lambda\tp H(Tx) \in \{0,1\}$ and $ATx \in \K[x_1,x_2]^n$. Furthermore, $\mu\tp H(Tx) \in K[x_1]$ 
only has terms of degree greater than $\deg \lambda\tp H(Tx)$, and hence no terms of degree less
than $\lambda\tp H(Tx)\; (\in \{0,1\})$ itself. Thus substituting $x = T^{-1} x$ gives the desired results.
\end{proof}

\noindent
Notice that in the case where $H$ is power linear and $\rk \jac H = 1$ in theorem \ref{posresnh}, 
we can even get $T^{-1} H(Tx) \in k[x_1]^n$ in (iii) of theorem \ref{tri}, namely by taking 
$\lambda\tp$ in the row space of $\jac H$. This is similar to the case where $H$ is power linear and 
$\rk \jac H = 2$ in theorem \ref{posres}, in the proof of which $T$ is taken such that 
$T^{-1} H(Tx) \in k[x_1,x_2]^n$ in (iii) of theorem \ref{tri}. 
It is however not always possible to take $T$ such that $T^{-1} H(Tx) \in k[x_1,x_2]^n$ 
in (iii) of theorem \ref{tri} when $H$ is power linear and $\rk \jac H = 2$ in theorem \ref{posresnh}, 
which the reader may show by taking e.g.\@ $H = \big(0,x_1^d, x_1^{d-1},(x_2+x_3)^d\big)$. 

Theorems \ref{posres} and \ref{posresnh} contain positive results with estimates on $\rk \jac H$, but for 
power linear $H$ only. Theorem \ref{posresrk} below however comprises two results with estimates on 
$\rk \jac H$, without the requirement that $H$ is power linear. Furthermore, the homogeneous 
counterexamples \eqref{n4} and 
\eqref{n5} later in this article show that the estimates in theorem \ref{posresrk} cannot be improved, 
even if we have the extra condition that $H$ is homogeneous.

\begin{theorem} \label{posresrk}
Assume that $H \in \K[x]^n$ has degree $d$, such that $H(0) = 0$ and $\jac H$ is nilpotent. 
If $\rk \jac H = 1$ or $\rk \jac H = 2 = d$, then $H$ is of the form \eqref{yantri}.

Furthermore, if we replace \eqref{yantri} by its equivalent in {\upshape (i)} of theorem \ref{tri}, 
then the condition $H(0) = 0$ is no longer necessary.
\end{theorem}

\begin{proof}
We show the equivalent property in (i) of theorem \ref{tri} instead of \eqref{yantri}.
The case $\rk \jac H \le 1$ follows from a corresponding strong nilpotence result in 
$(2) \Rightarrow (3)$ of \cite[Th.\@ 4.2]{1203.6615}, and the equivalence of strong nilpotence 
and the property in (i) of theorem \ref{tri}, which is proved in 
\cite{MR1412743}. 

So assume that $\rk \jac H = 2 = d$ and suppose without loss of generality that $H(0) = 0$.
The additional claim that the diagonal is zero follows from the nilpotency of $\jac H$, so we do
not need to worry about that any more.
By lemma \ref{hrd2} below, there exists a $T \in \GL_n(\K)$ such that for $\tilde{H} := T^{-1}H(Tx)$,
we have one of the following cases, which we treat individually.
\begin{itemize}

\item $\tilde{H} \in \K[x_1,x_2]^n$. \\
Then by theorem \ref{posresnh}, $(\tilde{H}_1, \tilde{H_2})$ has the property in (i) of theorem 
\ref{tri}. Hence we can choose $T$ such that $\jac_{x_1,x_2} (\tilde{H}_1, \tilde{H_2})$ is lower 
triangular. It follows that $\jac \tilde{H}$ is lower triangular as well, which is the property in
(i) of theorem \ref{tri}.

\item $\tilde{H}_3 = \tilde{H}_4 = \cdots = \tilde{H}_n = 0$. \\
Then by \cite[Th.\@ 7.2.25]{MR1790619}, we have
$$
(\tilde{H}_1, \tilde{H_2}) = (b g(a x_1 - b x_2) + d, a g(a x_1 - b x_2) + c)
$$
where $a,b,c,d \in \K[x_3,x_4, \ldots, x_n]$ and $g$ is an univariate polynomial over 
$\K[x_3,x_4, \ldots, x_n]$. Hence $a \tilde{H}_1 - b \tilde{H}_2 \in \K[x_3,x_4, \ldots, x_n]$.
Using that $\deg (\tilde{H}_1, \tilde{H}_2) = 2$, we see that either $g$ is constant or both $a$ and $b$ 
are constant. 

In both cases, there exists a nontrivial $\K$-linear combination of $\tilde{H}_1$ and
$\tilde{H_2}$ which is contained in $\K[x_3,x_4, \ldots, x_n]$. By choosing $T$ appropriate,
we can get $\tilde{H_2} \in \K[x_3,x_4, \ldots, x_n]$, in which case 
$\jac \tilde{H}$ is upper triangular. By a subsequent conjugation of $\tilde{H}$
with the map $(x_n, x_{n-1}, \ldots, x_2, x_1)$, we get the desired lower triangular form
of the Jacobian, which gives the property in (i) of theorem \ref{tri}.

\item $\tilde{H}_2 = \tilde{H}_3^2 \ne 0$ and
$\tilde{H}_4 = \tilde{H}_5 = \cdots = \tilde{H}_n = 0$. \\
If $\tilde{H}_3 \in \K[x_4, x_5, \ldots, x_n]$, then $\jac \tilde{H}$ is upper triangular, and
a subsequent conjugation of $\tilde{H}$ with the map $(x_n, x_{n-1}, \ldots, x_2, x_1)$ gives
the desired result. So assume that $\tilde{H}_3 \notin \K[x_4,x_5, \ldots, x_n]$.
Using polynomial extension of scalars on the case $n=3=d+1$ of theorem \ref{posresnh}, it follows that 
there exists a $\tilde{T} \in \GL_3\big(\K(x_4, x_5, \ldots, x_n)\big)$ 
such that $\jac_{\tilde{x}} \big(\tilde{T}^{-1}(\tilde{H}_1,\tilde{H}_2,\tilde{H}_3)|_{\tilde{x} = 
\tilde{T}\tilde{x}}\big)$ is lower triangular
with zeroes on the diagonal, where $\tilde{x} = x_1, x_2, x_3$. 

By clearing denominators in the first row of $\tilde{T}^{-1}$, we
see that there exists a nonzero $\lambda \in \K[x_4, x_5, \ldots, x_n]^3$ such that
$\lambda_1 \tilde{H}_1 + \lambda_2 \tilde{H}_2 + \lambda_3 \tilde{H}_3 \in \K[x_4, x_5, \ldots, x_n]$.
Since $\tilde{H}_2$ and $\tilde{H}_3$ have different positive degrees with respect to $\tilde{x}$, it
follows that $\lambda_1 \ne 0$ and that $\tilde{H}_1 \in \K[\tilde{H}_3, x_4, x_5, \ldots, 
\allowbreak x_n]$.

Now take $S \in \GL_n(\K)$ such that the $i$-th row of $S^{-1}$ equals
$e_i\tp$ for all $i \ge 4$ and the third row of $S^{-1}$ equals $\jac \tilde{H}_3$. Then
only the first three components of $S^{-1} \tilde{H}(Sx)$ are nonzero, and we have 
$\tilde{H}_3(Sx) = x_3$ and $(Sx)_i = x_i$ for all $i \ge 4$. Consequently, the first three 
components of $S^{-1} \tilde{H}(Sx)$ are contained in $\K[x_3, x_4, x_5, \ldots, x_n]$.
Hence the Jacobian of $S^{-1} \tilde{H}(Sx)$ is upper triangular, and
a subsequent conjugation of $\tilde{H}$ with the map $(x_n, x_{n-1}, \ldots, x_2, x_1)$ gives
the desired lower triangular form of the Jacobian. This gives the property in (i) of theorem \ref{tri}. 
\qedhere
\end{itemize}
\end{proof}

\begin{lemma} \label{hrd2}
Assume that $H \in \K[x]^n$ has degree $2$, such that $H(0) = 0$ and $\rk \jac H \le 2$. 
Then there exists a $T \in \GL_n(\K)$ such that $\tilde{H} := T^{-1} H(Tx)$ has one of the 
three forms that are specified in the proof of theorem \ref{posresrk}.
\end{lemma}

\begin{proof}
We can choose $T$ such that $\tilde{H}_1, \tilde{H}_2, \ldots, \tilde{H}_r$ have 
linearly independent quadra\-tic parts over $K$,
$\tilde{H}_{r+1}, \tilde{H}_{r+2}, \ldots, \tilde{H}_s$ are linear forms which are independent over $K$, and 
$\tilde{H}_{s+1} = \tilde{H}_{s+2} = \cdots = 0$. If $s \le 2$, then $\tilde{H} = T^{-1} H(Tx)$
has the second form in the proof of theorem \ref{posresrk}, so assume that $s \ge 3$. 
We distinguish three cases.
\begin{itemize}
 
\item $r \le 1$. \\
Then $\tilde{H}_2$ and $\tilde{H}_3$ are linear forms which are independent over $K$. 
Hence we can take $S \in \GL_n(\K)$
such that the first two rows of $S^{-1}$ are $\jac \tilde{H}_2$ and $\jac \tilde{H}_3$.
By the chain rule, $\jac \big(\tilde{H}_2(Sx)\big) = e_1\tp$ and $\jac \big(\tilde{H}_3(Sx)\big) = e_2\tp$, 
so $\tilde{H}_2(Sx) = x_1$ and $\tilde{H}_3(Sx) = x_2$. Hence $\tilde{H}(Sx) \in \K[x_1,x_2]^n$ and
$S^{-1} \tilde{H}(Sx) = (TS)^{-1} \bcdot H((TS)x)$ satisfies the first form in the proof of theorem 
\ref{posresrk}.

\item $r \ge 3$. \\
Since $\rk \jac \tilde{H} = 2$, the rows of $\jac(\tilde{H}_1, \tilde{H}_2, \tilde{H}_3)$ are linearly 
dependent over $\K(x)$ and hence also over $K[x]$. By looking at leading homogeneous parts, we see that
$\rk \jac (\bar{H}_1, \bar{H}_2, \bar{H}_3) \le 2$, where $\bar{H}_i$ is the leading and quadratic 
homogeneous part of
$\tilde{H}_i$ for each $i \le 3$. By \cite[Th.\@ 4.3.1]{homokema}, there exists linear forms $p, q$ such 
that $\bar{H}_1, \bar{H}_2, \bar{H}_3$ are linearly dependent over $\K$ of $p^2$, $pq$ and $q^2$. 
Furthermore, $p$ and $q$ are independent over $K$, and $p^2$, $pq$ and $q^2$ are in turn linearly 
dependent over $\K$ of $\bar{H}_1, \bar{H}_2, \bar{H}_3$. Thus there exists an $L \in \GL_3(\K)$ such that 
$L\big(\bar{H}_1, \bar{H}_2, \bar{H}_3\big) = (p^2, pq, q^2)$. 

Take $S \in \GL_n(\K)$ such that the
first two rows of $S^{-1}$ are $\jac p $ and $\jac q$, in that order. Then
$L\big(\bar{H}_1(Sx), \bar{H}_2(Sx), \bar{H}_3(Sx)\big) = (x_1^2, x_1 x_2, x_2^2)$.
The minor determinants of size $2$ of $\jac_{x_1,x_2}(x_1^2, x_1 x_2, x_2^2)$ are
$2x_2^2$, $4x_1x_2$ and $2x_1^2$, which are also linearly independent over $\K$. It follows that
$$
\det \jac_{x_1,x_2,x_i} \big(L\big(\tilde{H}_1(Sx), \tilde{H}_2(Sx), \tilde{H}_3(Sx)\big)\big) \ne 0
$$
holds if $i \ge 3$ and the last column of the Jacobian matrix on the left hand side, which can only be 
constant, is nonzero. Hence $L\big(\tilde{H}_1(Sx), \allowbreak
\tilde{H}_2(Sx), \tilde{H}_3(Sx)\big) \in \K[x_1,x_2]^3$. Since 
the first two rows of its Jacobian are linearly independent over $K$ and $L$ is 
invertible, $\tilde{H}(Sx) \in \K[x_1,x_2]^n$ holds as well. So
$S^{-1} \tilde{H}(Sx) = (TS)^{-1} \bcdot H((TS)x)$ satisfies the first form in the proof of 
theorem \ref{posresrk}.

\item $r = 2$. \\
If $s \ge 4$, then we can proceed as in the case $r \le 1$, but with $\tilde{H}_3$ and $\tilde{H}_4$
instead of $\tilde{H}_2$ and $\tilde{H}_3$. So assume that $s = 3$.

Since multiplication of the third row of $\jac \tilde{H}$ by $2\tilde{H}_3$ does not change the rank
of $\jac \tilde{H}$, we have $\rk \jac (\tilde{H}_1,\tilde{H}_2,\tilde{H}_3^2) \le 2$. Let 
$\bar{H}_i$ be the leading homogeneous part of $\tilde{H}_i$ for each $i \le 3$.
If $\bar{H}_3^2$ is linearly independent over $\K$ of $\bar{H}_1$ and $\bar{H}_2$, then
we can proceed as in the case $r \ge 3$ to obtain that $\tilde{H}_i(Sx) \in \K[x_1,x_2]$ for each 
$i \ne 3$ and $\tilde{H}_3(Sx)^2 \in \K[x_1,x_2]$ for some $S \in \GL_n(\K)$. So $S^{-1} \tilde{H}(Sx)
 = (TS)^{-1} H((TS)x)$ satisfies the first form in the proof of theorem \ref{posresrk} in that case.

So assume that $\bar{H}_3^2$ is linearly dependent over $\K$ of $\bar{H}_1$ and $\bar{H}_2$.
Then we can choose $T$ such that $\bar{H}_2 = \bar{H}_3^2$. If the linear part of 
$\tilde{H}_2$ is dependent over $K$ of $\tilde{H}_3$, then we can choose $T$ such that even $\tilde{H}_2 = 
\tilde{H}_3^2$. Since $s = 3$, we see that $\tilde{H} = T^{-1} H(Tx)$ has
the third form in the proof of theorem \ref{posresrk} in that case.

So assume that the linear part of 
$\tilde{H}_2$ is independent over $K$ of $\tilde{H}_3$.  Then $\tilde{H}_2 - \tilde{H}_3^2$ and $\tilde{H}_3$ are 
linear forms which are independent over $K$. Since $\jac (\tilde{H}_2 - \tilde{H}_3^2) = 
\jac \tilde{H}_2 - 2 \tilde{H}_3 \jac \tilde{H}_3$, we can replace $\tilde{H}_2$ by 
$\tilde{H}_2 -\tilde{H}_3^2$ without affecting the Jacobian rank of $\tilde{H}$, and proceed as in the case 
$r \le 1$ to obtain that $e_1\tp$ and $e_2\tp$ are in the row space of $\jac \tilde{H}(Sx)$ for some 
$S \in \GL_n(\K)$. Hence $\tilde{H}(Sx) \in \K[x_1,x_2]^n$ and
$S^{-1} \tilde{H}(Sx) = (TS)^{-1} H((TS)x)$ satisfies the first form in the proof of theorem \ref{posresrk}.
\qedhere

\end{itemize}
\end{proof}

\section{Lemmas}

The lemmas in this section are required for the proofs that the counterexamples in the next section
are indeed counterexamples.

\begin{lemma} \label{alem}
Let $d \ge 1$ and $a_1, a_2, \ldots, a_{2d+2} \in \K^n$ be pairwise linearly independent. 
Suppose that for all $j \ge \min\{3,d^2\}$ and all $k$ with $3 \le k \le d+2$, the set
$\{a_j, a_k, a_{k+d}\}$ consist of two or three vectors which are linearly independent
(depending on whether $j \in \{k,k+d\}$ or not).

If 
\begin{equation} \label{powersum}
\sum_{i=1}^{2d+2} \lambda_i (a_i\tp x)^d = 0
\end{equation}
for some $\lambda_i \in \K$, not all zero, then $\lambda_1 \lambda_2 \ne 0$. 
\end{lemma}

\begin{proof}
Assume that \eqref{powersum} holds. Since $a_1$ and $a_2$ are linearly independent, we may assume 
without loss of generality that $\lambda_3 \ne 0$. If $d=1$, then $\lambda_1 \lambda_2 = 0$ implies
that either $a_1$ or $a_2$ is linearly dependent of $a_3$ and $a_4$, which is a contradiction.
Hence the following cases remain.
\begin{itemize}

\item \emph{$d=2$.} \\
Since $a_4, a_5$ and $a_6$ are linearly independent and $d = 2$, we may assume without loss of generality that 
$a_1, a_3, a_6$ are linearly independent vectors. Consequently, there exists a $b_1 \in \K^n$ such that
$b_1\tp a_1 = b_1\tp a_6 = 0 \ne b_1\tp a_3$. Applying $b_1\tp \partial$ on \eqref{powersum} gives
$$
\sum_{i=2}^5 \mu_i (a_i\tp x)^1 = 0
$$
where $\mu_i = 2\lambda_i b_1 \tp a_i$ for all $i$. Since $a_3, a_4, a_5$ are linearly independent and
$\mu_3 \ne 0$, we have $\mu_2 \ne 0$ as well. Hence $\lambda_2 \ne 0$. In a similar manner, 
$\lambda_1 \ne 0$ follows.

\item \emph{$d>2$.} \\
Since $a_3$, $a_{d+2}$ and $a_{2d+2}$ are linearly independent, there exists a $b_2 \in \K^n$ 
such that $b_2\tp a_{d+2} = b_2\tp a_{2d+2} = 0 \ne b_2\tp a_3$. Applying $b_2\tp\partial$ on 
\eqref{powersum} gives
$$
\sum_{i=1}^{d+1} \mu_i (a_i\tp x)^{d-1} + \sum_{i=d+3}^{2d+1} \mu_i (a_i\tp x)^{d-1}= 0
$$
where $\mu_i = d\lambda_i b_2 \tp a_i$ for all $i$. Since $\mu_3 \ne 0$, it follows by induction 
on $d$ that $\mu_1 \mu_2 \ne 0$. Hence $\lambda_1 \lambda_2 \ne 0$. \qedhere

\end{itemize}
\end{proof}

\begin{lemma} \label{eqs}
\begin{equation} \label{666eq}
\sum_{i=0}^d (-1)^i \binom{d}{i} (x_1 + i x_3)^d = \sum_{i=0}^d (-1)^i \binom{d}{i} (x_2 + i x_3)^d
\end{equation}
and if $d \ge 2$ and $\zeta_d in K$ is a primitive $d$-th root of unity, then
\begin{equation} \label{667eq}
\sum_{i=0}^{d-1} \zeta_d^i (\zeta_d^i x_1 + x_2 + x_3)^d
+ \sum_{i=0}^{d-1} \zeta_d^i (\zeta_d^i x_1 + x_2 - x_3)^d = 2 d^2 x_1^{d-1} x_2
\end{equation}
\end{lemma}

\begin{proof}
We first prove \eqref{666eq}. Assume that \eqref{666eq} holds when we replace $d$ by $d-1$.
By substituting $x_2 = x_1 + x_3$ on both sides, we obtain
\begin{align*}
\sum_{i=0}^{d-1} (-1)^i \binom{d-1}{i} (x_1 + i x_3)^{d-1} 
&= \sum_{i=0}^{d-1} (-1)^i \binom{d-1}{i} (x_1 + (i+1) x_3)^{d-1} \\
&= -\sum_{i=1}^{d} (-1)^i \binom{d-1}{i-1} (x_1 + i x_3)^{d-1}
\end{align*}
By $\binom{d}{i} = \binom{d-1}{i-1} + \binom{d-1}{i}$, both sides combine to
$$
0 = \sum_{i=0}^{d} (-1)^i \binom{d}{i} (x_1 + i x_3)^{d-1} = \frac1d \partial_1 
    \sum_{i=0}^d (-1)^i \binom{d}{i} (x_1 + i x_3)^d
$$
Hence the left hand side of \eqref{666eq} is contained in $\K[x_3]$. By a symmetry argument, 
\eqref{666eq} follows by induction on $d$, because the case $d=0$ is trivial.

Assume that $d \ge 2$ and that $\zeta_d \in K$ is a primitive $d$-th root of unity.
By substituting $x_2 = x_2 \pm x_3$ in
\begin{equation} \label{667heq}
\sum_{i=0}^{d-1} \zeta_d^i (\zeta_d^i x_1 + x_2)^d = d^2 x_1^{d-1} x_2
\end{equation}
we get \eqref{667eq}. So in order to prove \eqref{667eq}, it suffices to show \eqref{667heq}. 
This can be done as follows.
\begin{align*}
\sum_{i=0}^{d-1} \zeta_d^i (\zeta_d^i x_1 + x_2)^d 
&= \sum_{i=0}^{d-1} \zeta_d^i \sum_{j=0}^{d} 
      \binom{d}{j} (\zeta_d^i x_1)^j x_2^{d-j} \\
&= \sum_{j=0}^{d} \binom{d}{j} x_1^j x_2^{d-j} 
      \sum_{i=0}^{d-1} \zeta_d^{i(j+1)} \\
&= \binom{d}{d-1} x_1^{d-1} x_2^1 \sum_{i=0}^{d-1} 1 \nonumber \\
&= d^2 x_1^{d-1} x_2 \qedhere
\end{align*}
\end{proof}

\noindent
Notice that \eqref{667eq} is not true for $d=1$. The proof uses $\zeta_d \in K$ to obtain that
$\zeta_d^1, \zeta_d^2, \ldots, \zeta_d^{d-1}$ are principal as $d$-th roots of unity,
and additionally $d \ge 2$ to obtain that $\zeta_d^{d+1}$ is principal as a $d$-th root of unity.

\begin{lemma} \label{powercomb}
To write $x_1^{d-1} x_2$ as a linear combination of $x_1^d$ and other $d$-th powers of linear forms, 
at least $d$ such powers are necessary besides $x_1^d$.
\end{lemma}

\begin{proof}
The case $d=1$ is easy, so let $d \ge 2$ and suppose that $x_1^{d-1} x_2$ can be written as a linear 
combination of $x_1^d, (a_3\tp x)^d, (a_4\tp x)^d, \ldots, (a_{d+1}\tp x)^d$. Assume without loss of 
generality that $n \ge 2d+2$ and that the vectors $e_1, a_3, a_4, \ldots, a_{d+1}$ are pairwise 
linearly independent. By applying $\partial_2$ on this linear combination, we obtain that
$$
x_1^{d'} = \lambda_3 (a_3\tp x)^{d'} + \lambda_4 (a_4\tp x)^{d'} + \cdots + \lambda_{d'+2} (a_{d'+2}\tp x)^{d'}
$$
where $d' = d-1$. Take $a_1 = e_1$ and take $a_2$ linearly independent of $a_1, a_3, \ldots, a_{d'+2}$. 
Next, take $a_i$ linearly independent of $a_1,a_2,\ldots, a_{i-1}$ for all $i$ with $d'+ 3 \le i \le 2d'+2$. 
Then lemma \ref{alem} with $d$ replaced by $d'$ gives a contradiction.
\end{proof}

\section{Counterexamples}

We start with giving counterexamples $x + H$ to $\eqref{jc} \Rightarrow \eqref{adddruz}$ and
$\eqref{addkel} \Rightarrow \eqref{yantri}$, such that $H$ is homogeneous of degree $d \ge 3$
and $d \ge 2$ respectively. With known techniques, these counterexamples can be improved to 
counterexamples of the form $x + (Ax)^{*d}$.

\begin{theorem} \label{n4n5}
If $n = 4$ and $d \ge 3$, then
\begin{equation} \label{n4}
H = x_1^{d-3}\big(0, 0,  x_2 (x_1 x_3 - x_2 x_4), x_1 (x_1 x_3 - x_2 x_4)\big)
\end{equation}
is a homogeneous counterexample of degree $d$ to $\eqref{jc} \Rightarrow \eqref{adddruz}$.

If $n = 5$ and $d \ge 2$, then
\begin{equation} \label{n5}
H = \big(0,0,x_2^{d-1}x_4,x_1^{d-1}x_3-x_2^{d-1}x_5,x_1^{d-1}x_4\big)
\end{equation}
is a homogeneous counterexample of degree $d$ to $\eqref{addkel} \Rightarrow \eqref{yantri}$.

Furthermore, there exist a power linear counterexample to $\eqref{jc} \Rightarrow 
\eqref{adddruz}$ for each $d \ge 3$, and a power linear counterexample to 
$\eqref{addkel} \Rightarrow \eqref{yantri}$ for each $d \ge 2$.
\end{theorem}

\begin{proof}
Assume first that $n = 4$ and $H$ is as in \eqref{n4}. Since the components of $H$ are composed
of the invariants $x_1, x_2, x_1 x_3 - x_2 x_4$ of $x + H$, we see that $x + H$ is a quasi-translation,
i.e.\@ $x - H$ is the inverse of $x + H$. One can compute that the trailing principal minor matrix of
size $2$ of $(d-1) I_4 + (d-2)\bcdot (\jac H)|_{x=(1,0,0,0)} + (\jac H)|_{x=(1,c,0,0)}$ equals
$$ 
\left( \begin{array}{cc} d - 1 + c & -c^2 \\ d - 1 & d - 1 - c \end{array} \right)
$$
and that its determinant equals $c^2 (d-2) + (d-1)^2$. So if we take
$c = \frac{d-1}{\sqrt{d-2}} \I$, then 
$$
\det \Big( (d-1) I_4 + (d-2) (\jac H)|_{x=(1,0,0,0)} + (\jac H)|_{x=(1,c,0,0)} \Big) = 0
$$
which contradicts $\eqref{jc} \Rightarrow \eqref{adddruz}$.

Assume next that $n = 5$ and $H$ is as in \eqref{n5}. Then one can compute that
$$
\jac H = \left( \begin{array} {ccccc} 
0 & 0 & 0 & 0 & 0 \\
0 & 0 & 0 & 0 & 0 \\
* & * & 0 & b & 0 \\
* & * & a & 0 & -b \\
* & * & 0 & a & 0
\end{array} \right)
$$
for certain polynomials $a, b$.
The form on the right hand side does not change by substitution and adding copies of $\jac H$ with different 
substitutions, so $\sum_{i=1}^n (\jac H)|_{x=v_i}$ is nilpotent for all $v_i \in \K^n$. 
This gives \eqref{addkel}.

On the other hand, $(\jac H)|_{x_1=0} \cdot (\jac H)|_{x_2=0}$ is a lower triangular matrix
with diagonal $(0,0,x_1^{d-1} x_2^{d-1},-x_1^{d-1} x_2^{d-1}, 0)$, so
$(\jac H)|_{x_1=0} \cdot (\jac H)|_{x_2=0}$ is not nilpotent. By \cite{MR1412743}, 
we see that $H$ is a counterexample to $\eqref{addkel} \Rightarrow \eqref{yantri}$.

To obtain power linear counterexamples, we can use the concept of GZ-pairing in \cite{MR1621913}.
For that purpose, let $H$ be any of the above two maps. By \cite[Th.\@ 1.3]{MR1621913}, there exists
an $N > n$ and an $A \in \Mat_N(\K)$, such that $x + H$ and $X + (AX)^{*d}$ are GZ-paired through
matrices $B \in \Mat_{n,N}(\K)$ and $C \in \Mat_{N,n}(\K)$, where $X = (x_1, x_2, \ldots, x_N)$. 
Take $M \in \Mat_{N,N-n}(\K)$ such that the columns of $M$ form a basis
of $\ker B$ and define $\bar{T} = (C \mid M)$. 

Then one can show that $\bar{T}$ is as in the proof of \cite[Th.\@ 2]{MR2108467}, with 
$F = X + (AX)^{*d}$ and $F_1 = (x + H, \ldots)$. 
Now one can use similar techniques as in the proof of theorem \ref{posres}
to obtain that $(AX)^{*d}$ is a counterexample as well as $H$, or use the following invariance results
for GZ-pairing. The GZ-invariance of \eqref{jc} follows from \cite[Th.\@ 1.3 (9)]{MR1621913} and
that of \eqref{yantri} from \cite[Th.\@ 3 (2)]{MR2419134}. The GZ-invariance of \eqref{adddruz}
and \eqref{addkel} can be proved with techniques in the proof \cite[Th.\@ 2.4]{MR1621913}.
\end{proof}

\begin{example}
Let $x =(x_1,x_2,\ldots,x_5)$ and $X =(x_1,x_2,\ldots,x_{13})$. Take $H$ as in \eqref{n5} and 
\begin{multline*}
G = \big(0,0,(x_4-x_1)^3,(x_4+x_1)^3,x_4^3,(x_4-x_2)^3,(x_4+x_2)^3, \\
(x_3-x_1)^3,(x_3+x_1)^3,x_3^3,(x_5-x_2)^3,(x_5+x_2)^3,x_5^3\big)
\end{multline*}
Then $\ker \jac_x G$ is trivial and $6H = BG$, where 
$$
B = \left( \begin{array}{ccccccccccccc}
\,1 & \,0 & \,0 & \,0 & \,0 & \,0 & \,0 & \,0 & \,0 & \,0 & \,0 & \,0 & \,0 \\
\,0 & \,1 & \,0 & \,0 & \,0 & \,0 & \,0 & \,0 & \,0 & \,0 & \,0 & \,0 & \,0 \\
\,0 & \,0 & \,0 & \,0 & \!\!\!\!-2\!\! & \,1 & \,1 & \,0 & \,0 & \,0 & \,0 & \,0 & \,0 \\
\,0 & \,0 & \,0 & \,0 & \,0 & \,0 & \,0 & \,1 & \,1 & \!\!\!\!-2\!\! 
& \!\!\!\!-1\!\! & \!\!\!\!-1\!\! & \,2 \\
\,0 & \,0 & \,1 & \,1 & \!\!\!\!-2\!\! & \,0 & \,0 & \,0 & \,0 & \,0 & \,0 & \,0 & \,0
\end{array} \right)
$$
has full rank. Hence there exists a matrix $C$ such that $BC = I_5$.
Consequently, $x + H$ and $X + \frac16G(BX)$ are GZ-paired through $B$ and $C$.
\end{example}

\noindent
In the next theorem, we give threedimensional counterexamples $F = x + H$ to $\eqref{yantri} \Rightarrow 
\eqref{druztri}$ and $\eqref{druztri} \Rightarrow \eqref{dittotri}$, such that $H$ is homogeneous of 
degree $d \ge 3$. The techniques in the proof of the previous theorem to get counterexamples of the 
form $F = x + (Ax)^{*d}$ do not work, so we improve our counterexamples to that form by hand.

\newpage

\begin{theorem}
Assume that $d \ge 2$, and either
\begin{equation} \label{666}
\left( \begin{array}{c} 
          H_1 \\ H_2 \\ H_3 \\ 
          H_4 \\ \vdots \\ H_{d+2} \\ 
          H_{d+3} \\ H_{d+4} \\ \vdots \\ H_{2d+2} 
\end{array} \right) := \left(\begin{array}{c}
          0 \\ \nu x_1^d \\ x_1^d - x_2^d \\
          (x_1 + 2 x_3)^d \\ \vdots \\ (x_1 + d x_3)^d \\
          (x_2 + x_3)^d \\ 
          (x_2 + 2 x_3)^d \\ \vdots \\ (x_2 + d x_3)^d \\
\end{array} \right)
\end{equation}
or
\begin{equation} \label{667}
\left( \begin{array}{c} 
          H_1 \\ H_2 \\ H_3 \\ 
          H_4 \\ \vdots \\ H_{d+2} \\ 
          H_{d+3} \\ H_{d+4} \\ \vdots \\ H_{2d+2} 
\end{array} \right) := \left(\begin{array}{c}
          0 \\ \nu x_1^d \\ x_1^{d-1} x_2\\
          (\zeta_d x_1 + x_2 + x_3)^d \\ \vdots \\ (\zeta_d^{d-1} x_1 + x_2 + x_3)^d \\
          (x_1 + x_2 - x_3)^d \\ 
          (\zeta_d x_1 + x_2 - x_3)^d \\ \vdots \\ (\zeta_d^{d-1} x_1 + x_2 - x_3)^d
\end{array} \right)
\end{equation}
for some $\nu \in \K$, where $\zeta_d$ is a primitive root of unity of $\K$ in the case of 
\eqref{667}. Then $2d + 2 \ge 6$,
\begin{equation} \label{666pl}
d(x_1 + x_3)^d = H_3 + \sum_{i=2}^d (-1)^i \binom{d}{i} H_{i+2}
                 - \sum_{i=1}^d (-1)^i \binom{d}{i} H_{i+d+2}
\end{equation}
in the case of \eqref{666} and
\begin{equation} \label{667pl}
(x_1 + x_2 + x_3)^d = 2d^2 H_3  - \zeta_d H_4 - \zeta_d^2 H_5 - \cdots - \zeta_d^{2d-1} H_{2d+2}
\end{equation}
in the case of \eqref{667}, and there exists a $T \in \GL_{2d+2}(\K)$ such that
$T(H(T^{-1}x))$ is power linear if $n = 2d+2$. 

If $3 \le n \le 2d+2$, then $H = (H_1, H_2, \ldots, H_n)$ is of the form \eqref{yantri} 
and we have the following.
\begin{enumerate}[\upshape (i)]

\item If $H$ is of the form \eqref{druztri}, then $c_1$ and $c_2$ are linearly independent linear combinations
of $e_1$ and $e_2$. 

\item $H$ is of the form \eqref{druztri}, if and only if either $H$ is as in \eqref{666} or
$H$ is as in \eqref{667} with $H_2  = 0 = d-2$.

\item $H$ is of the form \eqref{dittotri}, if and only if $H$ is as in \eqref{666} with $H_2 \ne 0$.

\end{enumerate}
\end{theorem}

\begin{proof}
Since $\jac H$ is lower triangular with zeroes on the diagonal,
it follows from (i) of theorem \ref{tri} that $H$ is of the form \eqref{yantri}. By \eqref{666eq} and
\eqref{667eq} in lemma \ref{eqs}, we get \eqref{666pl} and
\eqref{667pl} respectively. So $H$ is a linear triangularization of a power linear map if 
$n = 2d+2$.

In the case of \eqref{666}, set
\begin{align*}
a_{2+i}\tp x &:= x_1 + i x_3  & a_{d+2+i}\tp x &:= x_2 + i x_3 
\intertext{for $i = 1,2,\ldots,d$. In the case of \eqref{667}, set}
a_{2+i}\tp x &:= \zeta_d^{i-1} x_1 + x_2 + x_3 & a_{d+2+i}\tp x &:= \zeta_d^{i-1} x_1 + x_2 - x_3 
\end{align*}
for $i = 1,2,\ldots,d$. Then $H_i = (a_i\tp x)^d$ for all $i \ge 4$ and the left hand side of
\eqref{666pl} or \eqref{667pl} respectively is a multiple of $(a_3\tp x)^d$. Hence the linear 
span $S$ of $H_3, H_4, \ldots, H_n$ is contained in that of $(a_3\tp x)^d, (a_4\tp x)^d, \ldots, 
(a_{2d+2}\tp x)^d$.
\begin{enumerate}[\upshape (i)]

\item
\emph{Claim.} If $\mu\tp H$ and $\mu_2 H_2$ are both linearly dependent over $K$ of the same power of a 
linear form in $x_1$ and $x_2$ for some $\mu \in \K^n$, then $\mu$ is a linear combination of 
$e_1$ and $e_2$.

To prove the claim, assume that $\mu\tp H$ and $\mu_2 H_2$ are as above. Then there exists a nontrivial linear 
combination $a_2$ of $e_1$ and $e_2$ such that both $\mu\tp H$ and $\mu_2 H_2$ are linearly dependent of 
$(a_2\tp x)^d$. On account of $H_1 = 0$, we have
$$
0 x_4^d + (\mu_2 H_2 - \mu\tp H) + \mu_3 H_3 + \mu_4 H_4 + \cdots + \mu_n H_n = 0 
$$
Take $a_1 = e_4$. By \eqref{666pl} and \eqref{667pl} respectively, there exist a $\lambda \in \K^{2d+2}$ 
with $\lambda_1 = 0$, such that 
$$
\lambda_1 (a_1 \tp x)^d + \lambda_2 (a_2 \tp x)^d + \lambda_3 (a_3 \tp x)^d 
+ \cdots + \lambda_{2d+2} (a_{2d+2} \tp x)^d = 0
$$
Furthermore, there exists an injective linear map which maps $(\mu_3, \mu_4, \allowbreak \ldots, \mu_n)$
to $(\lambda_3, \lambda_4, \ldots, \lambda_{2d+2})$. By lemma \ref{alem} and $\lambda_1 = 0$, we have 
$\lambda = 0$. Thus $\mu_3 = \mu_4 = \cdots = \mu_n = 0$. So $\mu$ is a linear combination of 
$e_1$ and $e_2$ and the claim has been proved. 

Suppose that $H$ is of the form \eqref{druztri}. Since $c_j\tp b_i = 0$ for all $i \ge j \ge 1$, we have
\begin{equation} \label{c1H}
c_1\tp H = c_1\tp \sum_{i=1}^{n-1} (c_i\tp x)^d b_i = \sum_{i=1}^{n-1} c_1\tp b_i (c_i\tp x)^d = 0
\end{equation}
thus $c_1$ is a linear combination of $e_1$ and $e_2$ on account of the above claim. 
Using \eqref{c1H} again, we see that $c_1$ is linearly dependent of $e_1$ if $H_2 \ne 0$. 
Hence $(c_1\tp x)^d$ and $H_2$ are linearly dependent of the same power of a linear form in $x_1$ and $x_2$. 
By using $c_j\tp b_i = 0$ for all $i \ge j \ge 1$ again, we obtain that
\begin{equation} \label{c2H}
c_2\tp H = c_2\tp \sum_{i=1}^{n-1} (c_i\tp x)^d b_i = \sum_{i=1}^{n-1} c_2\tp b_i (c_i\tp x)^d = 
c_2\tp b_1 (c_1\tp x)^d
\end{equation}
It follows from the above claim that $c_2$ is a linear combination of $e_1$ and $e_2$ as well.

Suppose that $c_1$ and $c_2$ are linearly dependent. Then there exist a nontrivial linear combination $a_2$ of 
$e_1$ and $e_2$ such that both $c_1$ and $c_2$ are linearly dependent of $a_2$. Using the claim with 
$\mu_1 = \mu_2 = 0$, and $\mu\tp H = 0$ and $\mu\tp H = (a_2\tp x)^d$ respectively, 
we obtain $\dim S = n-2$ and $(a_2\tp x)^d \notin S$. 

The space $S^{*}$ generated 
by $(c_1\tp x)^d, (c_2 \tp x)^d, \ldots, (c_{n-1}\tp x)^d$, which contains $S$, is generated by 
$n-2$ powers of linear forms, namely $(a_2\tp x)^d, (c_3\tp x)^d, \allowbreak (c_4\tp x)^d, \ldots, 
(c_{n-1}\tp x)^d$. Hence $S^{*} \supseteq S$ and $\dim S^{*} \le n-2 = \dim S$. It follows that 
$S = S^{*}$. Since $(a_2\tp x)^d \notin S$ and $(a_2\tp x)^d \in = S^{*}$, we have 
a contradiction, so $c_1$ and $c_2$ are indeed linearly independent.

\item
If $H$ is as in \eqref{666}, then we can take
\begin{align*}
c_1 &= e_1      & b_1 &= \nu e_2 + e_3 \\
c_2 &= e_2      & b_2 &= -e_3 \\
c_i &= a_{i+1}  & b_i &= e_{i+1}
\intertext{for all $i > 2$, which shows that $H$ is of the form \eqref{druztri}. 
If $H$ is as in \eqref{667} with $H_2 = 0 = d-2$, then we can take}
c_1 &= e_1 + e_2 & b_1 &= \tfrac14 e_3 \\
c_2 &= e_1 - e_2 & b_2 &= -\tfrac14 e_3 \\
c_i &= a_{i+1}   & b_i &= e_{i+1}
\end{align*}
for all $i > 2$, which shows that again $H$ is of the form \eqref{druztri}. 

Conversely, suppose that $H$ is as in \eqref{667} and of the form \eqref{druztri}. 
By lemma \ref{powercomb}, at least $d$ powers of linear forms are necessary to write 
$H_2$ and $H_3$ as linear combinations of them if $H_2 = 0$, and
at least $d+1$ such powers otherwise. Now assume that $H_2 \ne 0$ or $d \ge 3$. Then there
are at least $3$ powers of linear forms necessary to write $H_2$ and $H_3$ as 
linear combinations of them. Hence there exist a linear combination $h$ of $H_2$ and $H_3$
which is not a linear combination of $(c_1\tp x)^d$ and $(c_2\tp x)^d$.

Since $\dim S^{*} \le n-1 < n$, there exists a nonzero $\mu \in \K^n$ such that
$$
\mu_1 (c_1\tp x)^d + \mu_2 (c_2\tp x)^d + \mu_3 h + \mu_4 H_4 + \mu_5 H_5 + \cdots + \mu_n H_n = 0
$$
By applying $\partial_3$ on both sides, we get
\begin{multline}
d \mu_4 (a_4\tp x)^{d'} + \cdots + d \mu_{d+2} (a_{d+2}\tp x)^{d'}
- d \mu_{d+3} (a_{d+3}\tp x)^{d'} \\ - d\mu_{d+4} (a_{d+4}\tp x)^{d'} - \cdots - 
d \mu_{2d+2} (a_{2d+2}\tp x)^{d'} = 0
\end{multline}
where $d' = d-1$ and $\mu_{n+1} = \mu_{n+2} = \cdots = \mu_{2d+2} = 0$. 
Take $a_1' = e_4$ and $a_2' = a_{d+3}$. Additionally set
$a_{i+1}' = a_{i+2}$ and $a_{i+d}' = a_{i+d+2}$ for all $i$ with $2 \le i \le d$. 
By lemma \ref{alem} with $d$ and $a$ replaced by $d'$ and $a'$ respectively,
we get $\mu_4 = \mu_5 = \cdots = \mu_{2d+2} = 0$. Hence $h$ is a linear combination of 
$(c_1\tp x)^d$ and $(c_2\tp x)^d$. Contradiction, so $H$ is not of the form \eqref{druztri}.

\item
Assume first that $H_2 \ne 0$. If $H$ is as in \eqref{666}, then we can take the $c_j$'s and 
the $b_i$'s as in (ii), and we have \eqref{dittotri}. If $H$ is as in \eqref{667},
then by (ii), $H$ is not of the form \eqref{druztri} and hence neither of the form \eqref{dittotri}.

Assume next that $H_2 = 0$ and that $H$ is of the form \eqref{dittotri}. By $H_1 = H_2 = 0$ and the 
fact that $c_2$ is linearly dependent of $e_1$ and $e_2$, we have $c_2\tp H = 0$. Consequently, $c_2\tp b_1 = 0$ 
on account of \eqref{c2H}. By definition of \eqref{dittotri}, we have $c_1 \tp b_i = c_2\tp b_i = 0$
for all $i$. Since $c_1$ and $c_2$ are linearly independent, we have a contradiction with the independence of
the $b_i$'s. \qedhere

\end{enumerate}
\end{proof}

\noindent
We can make non-homogeneous variants of \eqref{n4} and \eqref{n5} as follows.
In \eqref{n4}, we can replace $x_2$ by $1$, remove $H_2$, and replace $x_{i}$ by $x_{i-1}$ for all 
$i \ge 3$. In \eqref{n5}, we can replace $x_2^{d-1}$ by $x_1^{d-2}$, remove $H_2$, and replace 
$x_{i}$ by $x_{i-1}$ for all $i \ge 3$. In this manner, we get rid of the second coordinate, such that
both the dimension and the Jacobian rank respectively decrease by one, in return for abandoning 
homogeneity, just as with most of theorem \ref{posresnh} with respect to theorem \ref{posres}. 

The maps $H = (0,x_1^d - x_1^{d-1})$ and $H = (0,0,x_1^d - x_1^{d-1})$ are additional non-homogeneous
counterexamples to $\eqref{yantri} \Rightarrow \eqref{druztri}$ and $\eqref{druztri} \Rightarrow 
\eqref{dittotri}$ respectively. By comparing the counterexamples with the positive results of 
theorems \ref{posres}, \ref{posresnh} and \ref{posresrk}, we get the following four questions.

The first two questions are whether \eqref{adddruz} implies \eqref{addkel} in general and whether
\eqref{addkel} implies \eqref{yantri} in dimension three if $\jac H$ is nilpotent (if $F$ satisfies
\eqref{addkel}, then by \cite[Th.\@ 3.9]{MR2948624}), $JH$ gets nilpotent in additon if we compose 
$F$ with some linear map). In case $H$ is homogeneous, 
then the questions are whether \eqref{adddruz} implies \eqref{addkel} in general and whether
\eqref{addkel} implies \eqref{yantri} in dimension four, which are the last two questions.
By theorems \ref{posres} and \ref{posresnh}, the last and the second question respectively
have an affirmative answer when the degree is at most three.

\bibliographystyle{structconj}
\bibliography{structconj}

\end{document}